\documentclass[11pt]{article}

\usepackage{times}
\usepackage{graphicx}
\usepackage{psfrag}
\usepackage{amsmath,amsthm}
\usepackage{amsfonts}

\newtheorem{thm}{Theorem}[section]

\newtheorem{lem}[thm]{Lemma}

\theoremstyle{definition}
\newtheorem{defn}{Definition}[section]

\theoremstyle{remark}

\begin{document}

\parskip=10pt

\flushbottom 

\title{Packing-constrained point coverings} 

\author{
Veit Elser\\
Laboratory of Atomic and Solid State Physics\\
Cornell University, Ithaca, NY 14853-2501} 

\date{}

\maketitle

\begin{abstract}
In the packing-constrained point covering problem, PC$^2$, one seeks configurations of points in the plane that cannot all be covered by a packing arrangement of unit disks. We consider in particular the problem of finding the minimum number of points $N$ for which such a configuration exists and obtain the bounds $11\le N\le 55$. The disparity of these bounds is symptomatic, we believe, of the fact that PC$^2$ belongs in a higher complexity class than the standard packing and covering problems.
\end{abstract}

\section{Introduction}

When it seemed that all questions concerning the packing and covering properties of disks in the plane had been asked (and in many cases answered), Naoki Inaba \cite{Inaba} proposed a new one that, although easy to state, appears to be quite hard. Consider a two-player game where the first player places $N$ points in the plane. After all the points are placed, the second player tries to cover all the points with unit radius disks. Although the number of disks is unlimited, they may not overlap. The first player is declared the winner if the second player cannot cover all $N$ points. Problem: what is the minimum $N$ for which the first player has a winning strategy?

This problem combines elements of covering and packing optimization. The challenge is to design point sets of size $N$ that are difficult to cover with unit disks, when the disks are constrained to form a packing. A good indicator of the difficulty of this problem is the large gap between the lower and upper bounds, on the minimum size design that cannot be covered, that can be obtained with a reasonable effort. In this paper we establish the bounds $11\le N\le 55$. While both bounds are certainly poor, we submit them nevertheless as evidence of the difficulty of this problem.

\section{Lower bound}

The first player faces a challenge even when the second player is given a significant handicap.  Consider the handicap where the disks are required to form a close packing of the plane. The disk centers will always lie on the points of the hexagonal lattice $H$ with minimum distance 2, or more generally, a translation of $H$ by a vector $t$. The second player, in this restricted form of play, is limited to selecting the translation $t$ such that disks centered at $H+t$ cover as many of the first player's points as possible. Owing to the translation symmetry of $H$, both $t$ and the points played by the first player should be treated as elements of the fundamental domain $U=\mathbb{R}^2/H$. Clearly the minimum number of points $N'$ that cannot be covered by the second player with this handicap, given optimal play by the first player, is a lower bound on the number $N$ we seek for games without the handicap.

Our lower bound is based on properties of the \textit{interstitium}, the space  $I$ between the disks in a close packing. 
\begin{defn}
Let $D$ be the unit disk centered on the origin, then
\[
I(t)=\left(\mathbb{R}^2\setminus (H+t+D)\right)/H
\]
is the interstitium of a translated close packing. 
\end{defn}

\begin{thm}\label{10thm}
Any configuration of 10 points in the plane can be covered by a packing of unit disks.
\end{thm}
\begin{proof}
Let $P\subset U$ be a winning set of points for the first player when the second player has the handicap of being restricted to translates of close packings (such a set exists by theorem \ref{55thm}). Because $P$ is a winning set there is no translation $t$ such that $(H+t+D)/H$ covers every point of $P$, this last property being equivalent to $t\in I(p)$ for some $p\in P$. Thus every candidate translation $t$ is in the interstitium of some point of the wining set. But this is possible only if the interstitia $I(p)$ for $p\in P$ cover $U$. A necessary condition for this is
\[
|U|=\left|\cup_{p\in P}I(p)\right|\le \sum_{p\in P}|I(p)|
\]
from which we obtain
\[
|P|\ge \frac{|U|}{\max_{p\in P} |I(p)|}=\frac{2\sqrt{3}}{2\sqrt{3}-\pi}\approx 10.74.
\]
\end{proof}

\section{55-point configuration that cannot be covered}

In this section we show that at least one point in the configuration of 55 points shown in Figure 1 will not be covered in any packing arrangement of unit disks. Our construction exploits properties of the disk of radius $r=2/\sqrt{3}-1$ that fits snugly into the hole formed by three mutually tangent unit disks (Figure 2). We will refer to disks of this size as \textit{holes}.

\begin{figure}[!t]
\begin{center}
\includegraphics[width=4.in]{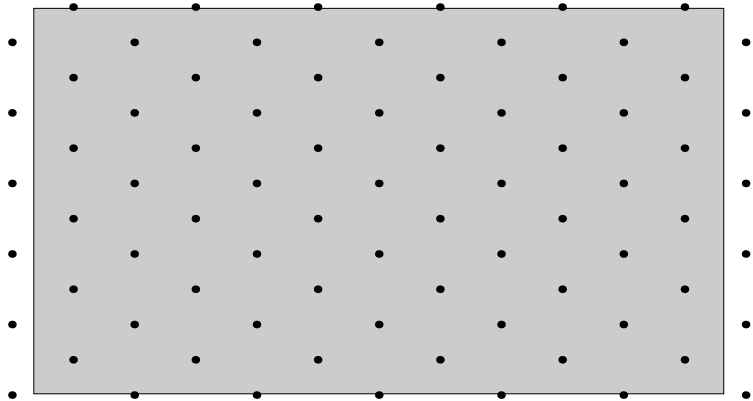}
\end{center}
\caption{The intersection of a rectangle (gray) of dimensions $2+4r$ and $1+3r$ and a hexagonal lattice of minimum distance $\sqrt{3}r$ gives a 55-point configuration that cannot be covered by a packing of unit disks ($r=2/\sqrt{3}-1$). Points near the horizontal edges of the rectangle are in its exterior.}
\end{figure}

\begin{defn}
A hole in a packing of unit disks is a disk of radius $r=2/\sqrt{3}-1$ disjoint from any of the unit disks.
\end{defn}

\begin{lem}\label{adjacentholelemma}
Let $D$ be a disk in a packing of unit disks and $A$ any arc of the circumference of $D$ having central angle $2\pi/6$, then a hole is tangent to $D$ somewhere along $A$.
\end{lem}
\begin{proof}
Consider configurations of holes tangent to $D$; their positions are constrained by the packing configuration of the other unit disks. Let $C$ be the subset of the circumference of $D$ where holes can be tangent. A disk $D'$ with distance $d$ from $D$ will exclude an open arc $E'$ from $C$ when $d<2r$. The central angle of $E'$ has the maximum value $2\pi/6$ when $d=0$. Two excluded arcs $E'$ and $E''$ cannot intersect because their intersection would be open and correspond to an open set of configurations where a hole is simultaneously tangent to three unit disks in a packing --- which is ruled out by choice of the hole radius $r$. The arcs excluded from $C$ are therefore disjoint and have maximum central angle $2\pi/6$. Intersecting $C$ with any arc of $D$ with central angle $2\pi/6$ will therefore include a point at which a hole is tangent.
\end{proof}

\begin{figure}[!t]
\begin{center}
\includegraphics[width=2.in]{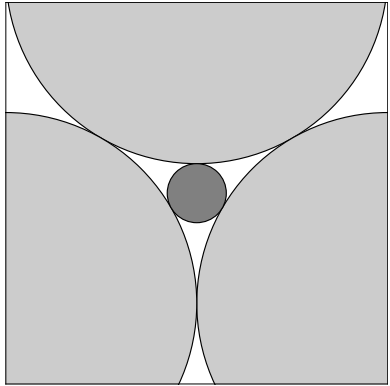}
\end{center}
\caption{A hole (dark gray) is the maximum size disk that fits in the interstitial space formed by three mutually tangent unit disks (light gray).}
\end{figure}

\begin{lem}\label{holeinrectanglelemma}
Let $R$ be a rectangle with sides $2+4r$ and $1+3r$, then in any packing of unit disks $R$ contains a hole.
\end{lem}
\begin{proof}
This argument relies on the construction shown in Figure 3, where $R$ is shown with vertices $A$, $B$, $C$ and $D$ and six holes are shown at special positions (centers) $E$, $F$, $G$, $H$, $I$ and $J$ within $R$. The dimensions of $R$ follow from tangency conditions satisfied at these special positions. The center of $R$ is $E$; a unit disk placed there will be tangent to holes at $F$, $G$, $H$ (and also $I$). Tangency of hole $F$ with rectangle side $BD$ determines the rectangle edge $d(A,B)=2+4r$. Holes $G$ and $H$ are also tangent to the unit disk placed at $E$. Requiring these to subtend angle $2\pi/6$ at $E$ and to be tangent to the sides of the rectangle determines the other dimension of $R$: $d(B,D)=d(G,H)+2r=1+3r$.

\begin{figure}[!t]
\begin{center}
\includegraphics[width=4.in]{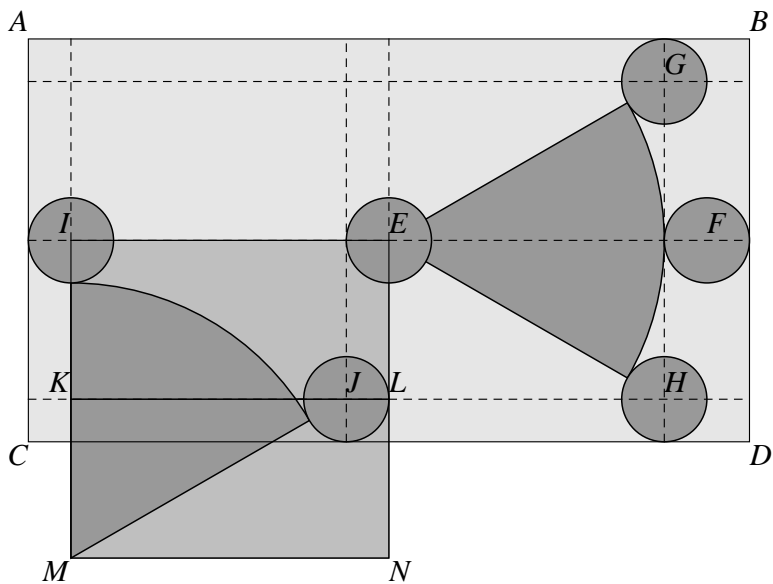}
\end{center}
\caption{Construction diagram used in the proof of lemma \ref{holeinrectanglelemma}.}
\end{figure}

We will also need the two smaller rectangles, $S'$ with vertices $I$, $E$, $K$ and $L$, and $S''$ with vertices $K$, $L$, $M$ and $N$. Tangencies determine $d(I,E)=1+r$ and $d(E,L)=(1+r)/2$. Rectangle $S'$ is congruent to $S''$; their union forms a square $S$ of side $1+r$. It is easily checked that $I$, $J$ and $M$ form the vertices of an equilateral triangle of side $1+r$. A unit disk centered at $M$ would thus be tangent to the holes at $I$ and $J$ and the angle subtended by the holes relative to $M$ will be $2\pi/6$.

Now consider any packing of unit disks and let $d$ be the minimum distance between $E$ and the disk centers. We may assume $d< 1+r$, since otherwise there is a hole at $E$. Without loss of generality we can assume the disk with distance $d$ has its center in the quadrant occupied by square $S$, and for convenience we even allow it to lie anywhere within $S$ (which contains the quarter circle of radius $1+r$). To complete the proof we need to show that a unit disk centered anywhere in $S$ will always admit an arc of its circumference for potential hole tangencies with central angle no less than $2\pi/6$ and such that the holes are always contained in $R$. By lemma \ref{adjacentholelemma} we then know that one of these holes will survive no matter how the other unit disks are packed.

First consider the case where the unit disk center lies in rectangle $S'$. The arc shown in Figure 3 and centered at $E$ is an extreme point in the case we are considering here. This arc is continuously related to arcs at all other centers in $S'$. The significance of $S'$ is that the arc endpoints, for disks centered in $S'$, are always determined by hole tangencies to sides $AB$ and $CD$ of $R$. Since the smallest arc arises when the centers of the two holes lie on a line perpendicular to those sides of $R$ (as for centers $G$ and $H$), the subtended angle will always satisfy the hypothesis of lemma \ref{adjacentholelemma}.

Finally, consider the only other case, where the unit disk center lies in $S''$. An extreme point for this case, with the disk center at $M$, is also shown in Figure 3. The endpoints of the arc of hole-disk tangencies are now determined by tangencies with the rectangle sides $AC$ and $CD$ (as for centers $I$ and $J$). The arc angle is an increasing function of the disk center, both as it moves from $M$ to $K$ and from $M$ to $N$. Since the hypothesis of lemma \ref{adjacentholelemma} was already satisfied at $M$ where the angle is at a minimum, it holds throughout $S''$.

\end{proof}

\begin{lem}\label{holesamplinglemma}
Let $H_d$ be a hexagonal lattice in the plane with minimum distance \mbox{$d<\sqrt{3} r$}, then a hole placed anywhere in the plane will contain at least one point of $H_d$.
\end{lem}
\begin{proof}
The hexagonal lattice $H_d$ with minimal distance $d$ is generated by translations $(d,0)$ and $(d/2, \sqrt{3} d/2)$. The covering radius of $H_d$, defined as the maximum distance between any point of the plane and the nearest point of $H_d$, has value $r_\mathrm{c}=d/\sqrt{3}$. A hole center placed anywhere in the plane will always be within distance $r_\mathrm{c}$ of a point of $H_d$. By choosing $r_\mathrm{c}<r$, or $d<\sqrt{3} r$, we ensure that the hole will contain a point of $H_d$.
\end{proof}

\begin{thm}\label{55thm}
The intersection of the rectangle of lemma \ref{holeinrectanglelemma} with the lattice of lemma \ref{holesamplinglemma}, suitably translated and rotated, gives a 55-point configuration that cannot be covered by non-overlapping unit disks .
\end{thm}
\begin{proof}
Let $X$ be the finite point set obtained by intersecting the rectangle $R$ in lemma \ref{holeinrectanglelemma} with the lattice $H_d$ of lemma \ref{holesamplinglemma}, to which we may apply arbitrary translations and rotations. By construction we know that, in any packing of unit disks, $R$ will contain a hole and this hole, as any hole, will contain a point of $H_d$ that is uncovered by unit disks. The relative translation and rotation used in the construction of the set $X$ shown in Figure 1 was chosen to minimize the number of points. To verify that the constraint $d<\sqrt{3} r$ on the minimum distance of $H_d$ can be satisfied we first construct $X$ with $d=\sqrt{3} r$ and check\footnote{This is how Figure 1 was constructed. The vertical separation of the points just outside $R$ is $11\sqrt{3} r/2\approx 1.4737$ while the vertical dimension of $R$ is $1+3r\approx 1.4641$.} that no points of $H_d$ lie on the boundary of $R$. Since this is the case for the 55-point configuration shown, the lattice can be compressed in scale (to satisfy $d<\sqrt{3} r$) without additional points entering $R$.
\end{proof}

\section{Discussion}

That our lower bound on $N$ is poor can be seen from the proof of theorem \ref{10thm}. The argument bounds $N$ by the minimum number of interstitium translates $I(t)$ required to cover the fundamental domain of the hexagonal lattice $U=\mathbb{R}^2/H$. Our proof uses only the area of $I(t)$ and does not exploit the fact that interstitia are rather inefficient covering shapes. Figure 4 shows the thinnest lattice covering, requiring the 25 translates $t\in H/5$. The thinnest covering we have found is a non-lattice set of 23 translations and it appears unlikely that this number can be reduced significantly. The best lower bound, based on the close packing handicap for the second player, is therefore likely to be 23 or slightly smaller, although proving this appears difficult.

\begin{figure}[!t]
\begin{center}
\includegraphics[width=3.in]{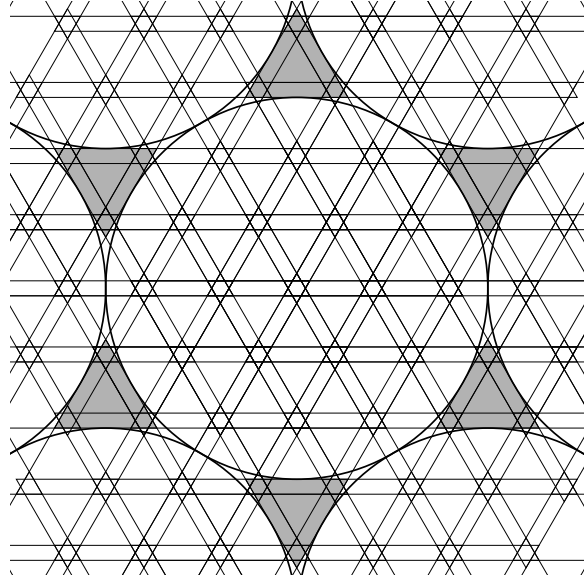}
\end{center}
\caption{A covering of the plane by the triangles (gray) internally tangent to the interstitium is a sufficient condition for the interstitium to cover the plane. This shows the covering obtained with the 25 translations generated by the lattice $H/5$.}
\end{figure}

There is a significant gap between the likely lower bound and our upper bound of 55. Depending on how strong of a handicap the close packing restriction is, the true value of $N$ will be closer to one end of this range or the other. The 55-point configuration appears to be 
non-optimal because not all of the points have the property that, when removed, the remaining points can be covered. We have not attempted to prove non-optimality because it involves the examination of many cases.

The hardness of the packing-constrained point covering problem, or PC$^2$, can be assessed by identifying the complexity class \cite{complexity} of a discrete variant. Finding a large set of binary codewords with a given minimum Hamming distance is a discrete variant of the standard packing problem. This is already quite hard, being an instance of the independent set problem which is known to be NP-complete. For PC$^2$ we propose an analogous variant. Fix the length of binary codewords. Let $P$ be the set of all codebooks of $N$ distinct codewords, and $C$ the set of all codebooks of codewords with Hamming distance greater than $2d$. Problem: what is the minimum $N$ such that there exists a $p\in P$ with the property that not all of its codewords are within Hamming distance $d$ of a codeword in some $c\in C$? The decision version of this problem (given a particular $N$) is equivalent to a generalization of Boolean satisfiability with two sets of variables, where a satisfying assignment for one set (there exists $c\in C$) is required for any truth assignment to the other subset (for any $p\in P$). Problems of this type are not in NP but belong to the superset PSPACE  of even harder problems. This observation suggests that computational efforts at improving the bounds for PC$^2$ may not get very far.


\begin{thebibliography}{99}
\bibitem{Inaba}  E. Pegg Jr. (2011), private communication.

\bibitem{complexity} M. R. Garey and D. S. Johnson (1979), \textit{Computers and Intractability: A Guide to the Theory of NP-Completeness}, W.H. Freeman. ISBN 0-7167-1045-5.

\end{thebibliography}
\end{document}